\documentclass{amsart}
\usepackage{amsfonts,amssymb,amscd,amsmath,enumerate,verbatim,calc}
\usepackage[all]{xy}

\newcommand{\CM}{Cohen-Macaulay}

\newcommand{\wrt}{with respect to}

\newcommand{\bF}{ \mathbf{f}}

\newcommand{\n}{\mathfrak{n} }
\newcommand{\m}{\mathfrak{m} }
\newcommand{\M}{\mathcal{M} }
\newcommand{\q}{\mathfrak{q} }

\newcommand{\R}{\mathcal{R} }
\newcommand{\Sc}{\mathcal{S} }
\newcommand{\Z}{\mathbb{Z} }
\newcommand{\bx}{\mathbf{x}}
\newcommand{\rt}{\rightarrow}

\newcommand{\ov}{\overline}

\newcommand{\wt}{\widetilde }

\newcommand{\image}{\operatorname{image}}

\newcommand{\reg}{\operatorname{reg}}
\newcommand{\grade}{\operatorname{grade}}
\newcommand{\depth}{\operatorname{depth}}
\newcommand{\rank}{\operatorname{rank}}
\newcommand{\ann}{\operatorname{ann}}

\newcommand{\cx}{\operatorname{cx}}
\newcommand{\projdim}{\operatorname{projdim}}
\newcommand{\Spec}{\operatorname{Spec}}

\newcommand{\Syz}{\operatorname{Syz}}
\newcommand{\Ass}{\operatorname{Ass}}
\newcommand{\Hom}{\operatorname{Hom}}
\newcommand{\Ext}{\operatorname{Ext}}
\newcommand{\Tor}{\operatorname{Tor}}

\theoremstyle{plain}

\newtheorem{theorem}{Theorem}[section]
\newtheorem{corollary}[theorem]{Corollary}
\newtheorem{lemma}[theorem]{Lemma}
\newtheorem{proposition}[theorem]{Proposition}

\theoremstyle{definition}
\newtheorem{definition}[theorem]{Definition}
\newtheorem{s}[theorem]{}
\newtheorem{remark}[theorem]{Remark}

\theoremstyle{remark}

\begin{document}

\title[Growth of Hilbert coefficients]{Growth of Hilbert coefficients of Syzygy modules}
\author{Tony~J.~Puthenpurakal}
\date{\today}
\address{Department of Mathematics, IIT Bombay, Powai, Mumbai 400 076}

\email{tputhen@math.iitb.ac.in}
 \begin{abstract}
Let $(A,\m)$ be a complete intersection ring of dimension $d$ and let $I$ be an $\m$-primary ideal. Let $M$ be a maximal \CM \  $A$-module.  For $i = 0,1,\cdots,d$,  let $e_i^I(M)$ denote the $i^{th}$ Hilbert -coefficient of $M$ with respect to $I$. We prove that for $i = 0, 1, 2$,   the function $j \mapsto e_i^I(\Syz_j^A(M))$ is of quasi-polynomial type with period $2$.
Let $G_I(M)$ be the associated graded module of $M$ with respect to $I$.    If $G_I(A)$ is Cohen-Macaulay and $\dim A \leq 2$ we also prove that the functions $j \mapsto \depth G_I(\Syz^A_{2j+i}(M))$ are eventually constant for $i = 0, 1$. 
Let $\xi_I(M) = \lim_{l \rt \infty}\depth G_{I^l}(M)$. 
Finally we prove that if $\dim A = 2$ and $G_I(A)$ is \CM \ then the functions $j \mapsto  \xi_I(\Syz^A_{2j + i}(M))$  are eventually constant for $i = 0, 1$.
\end{abstract}
 \maketitle
\section{introduction}
Let $(A,\m)$ be a Noetherian local ring of dimension $d$ and let $M$ be a finitely generated $A$-module of dimension $r$.  Let $I$ 
be an $\m$-primary ideal. Let $\ell(N)$ denote the length of an $A$-module $N$. The function $H^{(1)}_I(M,n) = \ell(M/I^{n+1}M)$ is called the \emph{Hilbert-Samuel} function of $M$ with respect to $I$. It is well-known that there exists a polynomial $P_I(M,X) \in \mathbb{Q}[X]$ of degree $r$ such that $P_I(M,n) = H^{(1)}_I(M,n)$ for $n \gg 0$. The polynomial $P_I(M,X)$ is called the Hilbert-Samuel polynomial of $M$ with respect to $I$.  We write
$$ P_I(M,X) = \sum_{i = 0}^{r}(-1)^ie_i^I(M)\binom{X+ r -i}{r-i}.$$
The integers $e_i^I(M)$ are called the $i^{th}$-Hilbert coefficent of $M$ with respect to $I$. The zeroth Hilbert coefficent $e_0^I(M)$ is called the \emph{multiplicity} of $M$ with respect to $I$. 

For $j \geq 0$ let $\Syz^A_j(M)$ denote the $j^{th}$ syzygy of $M$.
In this paper we investigate the function $j \mapsto e_i^I(\Syz^A_j(M))$ for $i \geq 0$.
It becomes quickly apparent that for reasonable answers  we need that the minimal resolution of $M$ should have some structure.  Minimal resolutions of modules over complete intersection rings have a good structure.  If $A = B/(f_1,\cdots,f_c)$ with $\mathbf{f} = f_1,\cdots,f_c$ a $B$-regular sequence and $\projdim_B M$ is finite then also the minimal resolution of $M$  has a nice structure. The definitive class of modules with a good structure theory of their minimal resolution is the class of modules with finite complete intersection dimension, see \cite{AGP}.   We are able to prove our results for a more restrictive class of modules than modules of finite CI-dimension.

\begin{definition}\label{gci}
 We say the $A$ module $M$ has finite GCI-dimension if there is a flat  local extension $(B,\n)$ of $A$  such that
 \begin{enumerate}
 \item
 $\m B = \n$.
 \item
  $B = Q/(f_1,\cdots,f_c)$ , where $Q$ is local and $f_1,\ldots,f_c$ is a $Q$-regular sequence.
\item
$\projdim_Q M\otimes_A B$ is finite.
\end{enumerate}
\end{definition}
We note that every finitely generated module over an abstract complete intersection ring has finite GCI dimension.  If $A = R/(f_1,\cdots,f_c)$ with $f_1,\cdots,f_c$ a $R$-regular sequence and $\projdim_R M$ is finite then also $M$ has finite GCI-dimension. We also note that if $M$ has finite GCI dimesnion then it has finite  CI-dimension. If $M$ has finite CI-dimension then the function $i \mapsto \ell(\Tor^A_i(M,k))$ is of quasi-polynomial type  with degree two. Set $\cx(M) = $ degree of this function $+ 1$. (See  \ref{degree} for degree of a function of quasi-polynomial type).

Let $G_I(A) = \bigoplus_{n \geq 0} I^n/I^{n+1}$ be the associated graded ring of $A$ \wrt \ $I$. Let $G_I(M) = \bigoplus_{n \geq 0} I^nM/I^{n+1}M$ be the associated graded module of $M$ \wrt \ $I$.   
Our main result is
\begin{theorem} \label{main}
Let $(A,\m)$ be a \CM \ local ring of dimension $d $ and let $M$ be a maximal \CM \ $A$-module.  Let $I$ be an $\m$-primary  ideal. Assume $M$ has finite GCI dimension. Then 
for $i = 0,1,2$,  the function  $j \mapsto e_i^I(\Syz^A_j(M))$ 
is of quasi-polynomial type with period two and degree $\leq \cx(M)-1$.
\end{theorem}

Next we consider the  asymptotic behavior  of depth of associated graded modules  of syzygy modules.  We prove
\begin{theorem}\label{main-depth}
Let $(A,\m)$ be a \CM \ local ring of dimension $\leq 2$ and let $M$ be a maximal \CM \ $A$-module.  Let $I$ be an $\m$-primary  ideal with $G_I(A)$ \CM. Assume $M$ has finite GCI dimension. Then
the functions $j \mapsto \depth G_I(\Syz^A_{2j}(M))$ and $j \mapsto \depth G_I(\Syz^A_{2j+1}(M))$ are constant for $j \gg 0$.
\end{theorem}

If $M$ is a \CM \ $A$-module and $I$ is $\m$-primary then it is known that $\depth G_{I^s}(M)$ is constant for all $s \gg 0$, see \cite[2.2]{Elias} (also see \cite[7.6]{Pu2}). Set $\xi_I(M) = \lim \depth G_{I^s}(M)$. We prove the following:
\begin{theorem}\label{main-depth-a}
Let $(A,\m)$ be a \CM \ local ring of dimension $2$ and let $M$ be a maximal \CM \ $A$-module.  Let $I$ be an $\m$-primary  ideal with $G_I(A)$ \CM. Assume $M$ has finite GCI dimension. Then
the functions $j \mapsto  \xi_I(\Syz^A_{2j}(M))$ and $j \mapsto  \xi_I(\Syz^A_{2j+1}(M))$ are constant for $j \gg 0$.
\end{theorem}

\s \emph{Dual Hilbert-Samuel function:} Assume $A$ has a canonical module $\omega$. The function $D^I(M,n) = \ell( \Hom_A(M, \omega/I^{n+1}\omega)$ is called the \emph{dual Hilbert-Samuel} function of $M$ \wrt \  $I$. In \cite{PuZ} it is shown that there exist a polynomial $t^I(M,z) \in \mathbb{Q}[z]$ of degree $d$ such that $t^I(M,n) = D^I(M,n)$ for all $n \gg 0$.
We write
$$ t^I(M,X) = \sum_{i = 0}^{d}(-1)^ic_i^I(M)\binom{X+ r -i}{r-i}.$$
The integers $c_i^I(M)$ are called the $i^{th}$- dual Hilbert coefficient of $M$ with respect to $I$. The zeroth dual Hilbert coefficient $c_0^I(M)$ is equal to $e_0^I(M)$, see \cite[2.5]{PuZ}.
We prove:
\begin{theorem}\label{main-dual}
Let $(A,\m)$ be a \CM  \ local ring  of dimension $d$, with a canonical module $\omega$ and let $I$ be an $\m$-primary ideal. Let $M$ be a maximal \CM \ $A$-module. Assume $M$ has finite GCI dimension.  Then
for $i = 0,1, \cdots, d$  the function  $j \mapsto c_i^I(\Syz^A_j(M))$ 
is of quasi-polynomial type with period two. If $A$ is a complete intersection then the degree of each of the above functions $\leq \cx(M)-1$.
\end{theorem}
Although  Theorem \ref{main-dual} looks more complicated than Theorem \ref{main},  its proof is
 considerably simpler.

Let $H^i(-)$ denote the $i^{th}$ local cohomology functor of $G_I(A)$ \wrt \ $G_I(A)_+ = \bigoplus_{n > 0}I^n/I^{n+1}$. Set
\[
\reg(G_I(M)) = \max\{ i+j \mid H^i(G_I(M))_j \neq 0 \},
\]
the \emph{regularity} of $G_I(M)$. Set
\[
a_i(G_I(M)) = \max \{ j \mid H^i(G_I(M))_j \neq 0 \}.
\]
Assume that the residue field of $A$ is infinite. Let $J$ be a minimal reduction of $I$. Say $I^{r+1} = JI^r$. Let $M$ be a maximal \CM \ $A$-module. Then it is well-known that $a_d(G_I(M))  \leq r - d $; see \cite[3.2]{T}(also see \cite[18.3.12]{BS}). We prove
\begin{theorem} \label{main-reg}
Let $(A,\m)$ be a \CM \ local ring of dimension $d \geq 2 $ and let $M$ be a maximal \CM \ $A$-module.  Let $I$ be an $\m$-primary  ideal. Assume $M$ has finite GCI dimension. Then the set
\[
\left\{ \frac{a_{d-1}(G_I(\Syz^A_i(M))}{i^{\cx(M)-1}}  \right\}_{i \geq 1 }
\]
is bounded.
\end{theorem}

Here is an overview of the contents of this paper. In section two we discuss many preliminary facts that we need. 
In section three we introduce a technique which is useful to prove our results.
In section four(five) we discuss the case when dimension of $A$ is one(two) respectively. The proof of Theorem  \ref{main} is divided in these two sections. The proof of Theorem \ref{main-depth} is in section five. In section six we prove Theorem \ref{main-depth-a}. Theorem \ref{main-dual} is proved in section seven. Finally in section eight we prove Theorem \ref{main-reg}.

\section{Preliminaries}
In this section we collect a few preliminary results that we need.

\s
\label{hilbcoeff}
 The \emph{Hilbert function} of $M$ with respect to $I$ is the function
\[
 H_{I}(M,n) =  \lambda(I^nM/I^{n+1}M)\quad \text{for all} \ n \geq 0.
\]
It is well known that the formal power series $\sum_{n \geq 0}H_{I}(M,n)z^n$
represents a rational function of a
special type:
\begin{equation*}
\sum_{n \geq 0}H_{I}(M,n) z^n = \frac{h_{I}(M ,z)}{(1-z)^{r}}\quad \text{where}
\ r = \dim M \ \text{and} \ h_{I}(M,z) \in\mathbb{Z}[z].
\end{equation*}
 It can be shown that
$e_{i}^{I}(M) = (h_I(M,z))^{(i)}(1)/i! $ for all $ 0 \leq i \leq r$.  It is convenient to set $e_{i}^{I}(M) = (h_I(M,z))^{(i)}(1)/i! $  even when $i \geq r$.
The number $e_{0}^{I}(M)$ is also called the \emph{multiplicity} of $M$ with respect
to $I$.

\textbf{I:} \textit{Superficial elements.}\\
For definition and basic properties of superficial sequences see \cite[p.\ 86–-87]{Pu1}.
The following result is well-known to experts. We give a proof due to lack of a reference.
\begin{lemma}\label{count-sup}
Let $(A,\m)$ be a \CM \ local ring of dimension $d \geq 1$ and let $I$ be an $\m$-primary ideal. Assume the residue field $k = A/\m$ is uncountable. Let $\{M_n \}_{n \geq 1}$ be a sequence of maximal \CM \ $A$-modules. Then there exists $\bx = x_1,\ldots,x_d \in I$ such that $\bx$ is an $M_n$-superficial sequence (\wrt \ $I$) for all $n \geq 1$.
\end{lemma}
\begin{proof}
We recall the construction of a superficial element for a  finitely generated module $M$ of dimension $r \geq 1$ (note $M$ need not be \CM). Let $\M$ be the maximal ideal of $G_I(A)$. Let
\[
\Ass^*(G_I(M)) = \{ P \mid P \in \Ass(G_I(M)) \ \text{and} \ P \neq \M \}.
\]
Set $L = I/I^2$ and $V = L/\m L = I/ \m I$. For $P \in \Spec G_I(A)$ let $P_1 = P \cap L$ and
$\ov{P_1} = (P_1 + \m I)/\m I$.
 If $P \neq \M$ then $P_1 \neq L$.  By Nakayama's Lemma it follows easily that  $\ov{P_1} \neq V$. If $k$ is infinite then 
 $$ S(M)  = V \setminus \bigcup_{P \in \Ass^*(G_I(M))} \ov{P_1} \neq \emptyset. $$
 Then $u \in I$ such that $\ov{u} \in S(M)$ is an $M$-superficial element with respect to $I$.
 
 Now assume that $k$ is uncountable and let $\mathcal{F} = \{M_n \}_{n \geq 1}$ be a sequence of maximal \CM \ $A$-modules. Set $M_0 = A$. Then 
 $$ S(\mathcal{F})  = V \setminus \bigcup_{n \geq 0} \bigcup_{P \in \Ass^*(G_I(M_n))} \ov{P_1} \neq \emptyset. $$
 If $u \in I$ such that $\ov{u} \in S(\mathcal{F})$ then $\ov{u} \in S(M_n)$ for all $n \geq 1$. Thus $u$ is $M_n$-superficial element with respect to $I$ for all $n \geq 0$. Set $x_1 = u$. If $d \geq 2$ then note that $\{ M_n/x_1 M_n \}$ is a sequence of maximal \CM \  $A/(x_1)$ modules. So by induction the result follows.
\end{proof}

\s\textbf{ Associated graded module and Hilbert function mod a superficial element: }\label{mod-sup-h} \\
Let $x \in I$
be $M$-superficial.  Set $N = M/xM$. Let $r = \dim M$.  The following is well-known cf., \cite{Pu1}.
\begin{enumerate}[\rm(1)]
  \item Set $b_I(M,z) = \sum_{i \geq 0}\ell\big( (I^{n+1}M \colon_M x)/ I^nM \big)z^n$. Since
  $x$ is $M$-superficial we have $b_I(M,z) \in \Z[z]$.
  \item $h_I(M,z) = h_I(N,z) - (1-z)^rb_I(M,z)$; cf., \cite[Corollary 10 ]{Pu1}.
  \item So we have
  \begin{enumerate}[\rm(a)]
    \item $e_i(M) = e_i(N)$ for $i = 0,\ldots, r-1$.
    \item $e_r(M) = e_r(N) - (-1)^r b_I(M,1)$.
  \end{enumerate}
  \item
   The following are equivalent
   \begin{enumerate}[\rm(a)]
     \item $x^*$ is $G_I(M)$-regular.
     \item $G_I(N) = G_I(M)/x^*G_I(M)$
     \item $b_I(M,z) = 0$
     \item $e_r(M) = e_r(N)$.
   \end{enumerate}
   \item
   \emph{(Sally descent)} If $\depth G_I(N) > 0$ then $x^*$ is $G_I(M)$-regular.
\end{enumerate}

\textbf{II:}\textit{Base change.} 
\s\label{AtoA'}
 Let $\phi \colon (A,\m) \rt (A',\m')$ be a flat local ring homomorphism with $\m A' = \m'$. Set $I' = IA'$ and if
 $N$ is an $A$-module set $N' = N\otimes A'$.
 It can be easily seen that

\begin{enumerate}[\rm (1)]
\item
$\lambda_A(N) = \lambda_{A'}(N')$.
\item
$\projdim_A N = \projdim_{A'} N'$.
\item
 $H^I(M,n) = H^{I'}(M',n)$ for all $n \geq 0$.
\item
$\dim M = \dim M'$ and  $\grade(K,M) = \grade(KA',M')$ for any ideal $K$ of $A$.
\item
$\depth G_{I}(M) = \depth G_{I'}(M')$.
\end{enumerate}

 \noindent The specific base changes we do are the following:

(i) If $M$ is a maximal \CM \ $A$-module with finite GCI dimension then by definition there exists a flat homomorphism $(A,\m) \rt (B,\n)$ with $\m B = \n$ and $B = Q/(f_1,\ldots,f_c)$ for some local ring $Q$ and a $Q$-regular sequence $f_1,\ldots, f_c$ such that $\projdim_Q M\otimes_A B < \infty$. In this case we set $B = A^\prime$.

(ii) If $k \subseteq k'$ is an extension of fields then it is well-known that there exists a flat  
local ring homomorphism $(A,\m) \rt (A',\m')$ with $\m A' = \m'$ and $A'/\m' = k'$.
We use this construction when the residue field $k$ of $A$ is finite or countably infinite. We take $k'$ to be any uncountable field containing $k$.

(iii) We can also take $A'$ to be the completion of $A$.
\begin{remark}\label{base-change}
Let $(A,\m)$ be a \CM  \ local ring and let $M$ be a maximal \CM \ $A$-module with 
finite GCI dimension. After doing the base-changes (i), (ii) and (iii) above we may assume that $A = Q/(f_1,\ldots,f_c)$ for some \emph{complete} \CM \ local ring $Q$ and a $Q$-regular sequence $f_1,\ldots, f_c$ such that $\projdim_Q M < \infty$. Furthermore we may assume that 
the residue field of $A$ is uncountable.
\end{remark}

\textbf{III:}\textit{Quasi-polynomial functions of period 2}. \\
 Let us recall that a function $f \colon \Z_{\geq 0} \rt \Z_{\geq 0}$ is said to be of \textit{ quasi-polynomial type} with period $g \geq 1$ if there exists polynomials $P_0, P_1,\ldots, P_{g-1} \in \mathbb{Q}[X]$ such that $f(mg + i) = P_i(m)$ for all $m \gg 0$ and $i = 0, \cdots, g-1$.

We need the following  well-known result regarding quasi-polynomials of period two.

\begin{lemma}\label{g2}
Let $f \colon \Z_{\geq 0} \rt \Z_{\geq 0}$. The following are equivalent:
\begin{enumerate}[\rm (i)]
\item
$f$ is of quasi-polynomial type with period $2$.
\item
$\sum_{n\geq 0} f(n)z^n = \frac{h(z)}{(1-z^2)^c}$ for some  $h(z) \in \Z[z]$ and $c \geq 0$.
\end{enumerate}
Furthermore if $P_0, P_1 \in Q[X]$ are polynomials such that $f(2m + i) = P_i(m)$ for all $m \gg 0$ and $i = 0, 1$ then $\deg P_i \leq c-1$. \qed
\end{lemma}

\textit{Convention:} We set the degree of the zero-polynomial to be $-\infty$. 

\begin{remark}\label{degree}
Let $f \colon \Z_{\geq 0} \rt \Z_{\geq 0}$ be of quasi-polynomial type with period $2$. Let $P_0, P_1 \in Q[X]$ be polynomials such that $f(2m + i) = P_i(m)$ for all $m \gg 0$ and $i = 0, 1$. Then set $\deg f = \max\{\deg P_0, \deg P_1 \}$. 
\end{remark}
As a consequence we get that

\begin{corollary}\label{g2-c}
Let $f \colon \Z_{\geq 0} \rt \Z_{\geq 0}$. If the function $g(n) = f(n) + f(n-1)$ is of quasi-polynomial type with period $2$ then so is $f$. Furthermore 
$\deg f = \deg g$.
\end{corollary}
\begin{proof}[Sketch of a proof]
Set $u(z) = \sum_{n \geq 0}f(n)z^n$. Then by our hypothesis and Lemma \ref{g2} we have
\[
(1+z)u(z) + f(0)z = \frac{h(z)}{(1-z^2)^c},
\]
for some  $h(z) \in \Z[z]$ and $c \geq 0$. An easy computation now shows that
\[
u(z) = \frac{r(z)}{(1-z^2)^c},
\]
for some  $r(z) \in \Z[z]$. Thus $f$ is of quasi-polynomial type with period two.

As $f(n) \leq g(n)$ we clearly get $\deg f  \leq \deg g$. As $g(2n) = f(2n) + f(2n-1)$ and $g(2n + 1) = f(2n+1) + f(2n)$ we get that $\deg g \leq \deg f$.
\end{proof}
The following result is also well-known:
\begin{lemma}\label{g2-growth}
Let $f \colon \Z_{\geq 0} \times \Z_{\geq 0} \rt \Z_{\geq 0}$ be a function such that
\[
\sum_{m,n\geq 0} f(m,n)z^mw^n = \frac{h(z,w)}{(1-z^2)^c(1-w)^{d}} 
\]
Then for $m,n \gg 0$ we have
\[
f(m,n) = \sum_{i = 0}^{d-1}(-1)^ie_i(m)\binom{n+d-1-i}{d-1-i},  
\]
where the functions $m \mapsto e_i(m)$ for $i = 0,\ldots, d-1$ are of  quasi-polynomial type with  period $2$ and degree $\leq c-1$. \qed
\end{lemma}

\textbf{IV:} \textit{Eisenbud operators}\\
Let $Q$ be a Noetherian ring  and let $\mathbf{f}= f_1,\ldots f_c$ be a regular sequence in $Q$.  Set $A = Q/(\mathbf{f})$.  Let $M$ be a finitely generated $A$-module with $\projdim_Q M$ finite. 
\s Let $\mathbb{F}:  \cdots F_n \rt \cdots F_1 \rt F_0 \rt 0$
 be a free resolution of $M$ as a $A$-module.

 Let $t_1,\ldots t_c \colon \mathbb{F}(+2)  \rt \mathbb{F} $ be the \emph{Eisenbud-operators}; see
 \cite[section 1.]{Eisenbud-80}. Then
 \begin{enumerate}
   \item $t_i$ are uniquely determined up to homotopy.
   \item $t_i, t_j$ commute up to homotopy.
 \end{enumerate}
 Let $T = A[t_1,\ldots,t_c]$ be a polynomial ring over $A$ with variables $t_1,\ldots,t_c$ of degree $2$.
Let $D$ be an $A$-module. The operators $t_j$ give well-defined maps
\[
t_j \colon \Ext^{i}_{A}(M,D) \rightarrow \Ext^{i+2}_{R}(M,D) \quad \ \text{for} \ 1 \leq j \leq c  \ \text{and all} \  i,
\]
\[
t_j \colon \Tor^{i+2}_{A}(M,D) \rightarrow \Tor^{i}_{R}(M,D) \quad \ \text{for} \ 1 \leq j \leq c  \ \text{and all} \  i.
\]
This turns $\Ext_A^*(M,D) = \bigoplus_{i \geq 0} \Ext^i_A(M,D)$ and $\Tor_A^*(M,D) = \bigoplus_{i \geq 0} \Tor^i_A(M,D)$  into  modules over $T$ (here we give an element $t \in \Tor^i_A(M,D)$ degree $-i$). Furthermore these structure depend only on $\bF$, are natural in both module arguments and commute with the connecting maps induced by short exact sequences.

\s 
Gulliksen, \cite[3.1]{Gulliksen},  proved that if $\projdim_Q M$ is finite then
$\Ext_A^*(M,D) $ is a finitely generated $T$-module. 
If $A$ is local and  $D = k$, the residue field of $A$, Avramov in \cite[3.10]{LLAV} proved a converse; i.e., if
$\Ext_A^*(M,k)$ is a finitely generated $T$-module then $\projdim_Q M$ is finite. For a more general result, see \cite[4.2]{AGP}.

\begin{definition}(with notation as above:) Assume $A$ is local with residue field $k$.
 Set $\cx M = \dim_T \Ext_A^*(M,k)$, the \emph{complexity} of $M$.
\end{definition}

We need the following result regarding the growth of lengths of certain Tor's. Recall a graded module $X$ over $T = A[t_1,\ldots, t_c]$ is said to be *-Artinian $T$-module if  every descending chain of
graded submodules of $X$ terminates.

\begin{proposition}\label{growth}
Let $(Q,\n)$ be a complete Noetherian ring  and let $\mathbf{f}= f_1,\ldots f_c$ be a regular sequence in $Q$.  Set $A = Q/(\mathbf{f})$ and $\m = \n/(\bF)$.  Let $M$ be a finitely generated $A$-module with $\projdim_Q M$ finite. Let $D$ be a non-zero $A$-module of finite length.  Let $\Ext_A^*(M,D) $ be a finitely generated $T = A[t_1,\ldots, t_c]$-module as above.  Then
\begin{enumerate}[\rm (1)]
\item
$\dim_T \Ext_A^*(M,D) \leq \cx(M)$.
\item
The function $n \mapsto \ell( \Ext^n_A(M, D))$ is of  quasi-polynomial type with  period $2$ and degree $\leq \cx(M)-1$. 
\item
$\Tor_A^*(M,D)$ is a *-Artinian $T$-module. Here $t \in \Tor^A_n(M, D)$ has degree $-n$.
\item
The function $n \mapsto \ell( \Tor^n_A(M, D))$ is of  quasi-polynomial type with  period $2$ and degree $\leq \cx(M)-1$. 
\end{enumerate}
\end{proposition}
\begin{proof}
(1) This follows from \cite[Theorem 5.3]{AGP}.

(2) This easily follows from (1).

For (3), (4) we use the following result. Let $E$ be the injective hull of $k$. Then
we have
\[
\Hom_A(\Tor^A_n(M, D), E) \cong \Ext^n_A(M, \Hom_A(D, E)). \tag{$\dagger$}
\]

(3) From $(\dagger)$ it follows that the Matlis dual of 
$\Tor_A^*(M,D)$ is  \\ $\Ext^*_A(M, \Hom_A(D, E))$. By Gulliksen's result we have that 
$\Ext^*_A(M, \Hom_A(D, E))$ is a finitely generated graded $T$-module. So by Matlis-duality \cite[3.6.17]{BH} we get that \\ $\Tor_A^*(M,D)$ is a *-Artinian $T$-module.

(4) This follows from (2) and $(\dagger)$. 
\end{proof}
As an immediate consequence we obtain:
\begin{corollary}\label{mult}
Let $(A,\m)$ be a \CM \ local ring of dimension $d$ and let $I$ be an $\m$-primary ideal. Let $M$ be a maximal \CM \ $A$-module.  If $M$ has finite GCI dimension over $A$ then the function
$i \mapsto e_0^I(\Syz^A_i(M))$ is of quasi-polynomial type with period $2$ and degree $= \cx(M) - 1$.
\end{corollary}
\begin{proof}
By Remark \ref{base-change} we may assume $A = Q/(\mathbf{f})$ where $Q$ is complete with infinite residue field, $\mathbf{f} = f_1,\cdots, f_c$ is a $Q$-regular sequence and 
$\projdim_Q M$ is finite.  Let $\mathbb{F}$ be a minimal resolution of $M$. 
Then $ \rank F_i = \ell(\Tor^A_i(M,k))$. So by \ref{growth} we get that the function $i \mapsto \rank F_i$ is quasi-polynomial of period two.
Set $M_i = \Syz^A_i(M)$.  The exact sequence
$0 \rt M_{i+1} \rt F_i \rt M_i \rt 0$ yields 
$e_0^I(M_i) + e^I_0(M_{i+1}) =  (\rank F_i) e_0^I(A)$. The result
now follows from \ref{g2-c}.
\end{proof}

We also prove:
\begin{proposition}\label{hilbSyz}
Let $(A,\m)$ be a \CM \ local ring of dimension $d$ and let $I$ be an $\m$-primary ideal. Let $M$ be a maximal \CM \ $A$-module. Assme $M$ has finite GCI dimension over $A$.
Fix $n \geq 0$. Then the function 
\[
i \mapsto \ell\left(\frac{I^n\Syz^{i}_A(M)}{I^{n+1}\Syz^{i}_A(M)}\right)
\]
is of quasi-polynomial type with period two and degree $\leq \cx(M)-1$.
\end{proposition}
\begin {proof}
By Remark \ref{base-change} we may assume $A = Q/(\mathbf{f})$ where $Q$ is complete with infinite residue field, $\mathbf{f} = f_1,\cdots, f_c$ is a $Q$-regular sequence and 
$\projdim_Q M$ is finite.  Let $\mathbb{F}$ be a minimal resolution of $M$. Set $M_i = \Syz^A_i(M)$.  Fix $n \geq 0$. The exact sequence
$0 \rt M_{i+1} \rt F_i \rt M_i \rt 0$ yields an exact sequence
\[
0 \rt \Tor^A_{i+1}(M, A/I^{n+1}) \rt M_{i+1}/I^{n+1}M_{i+1} \rt F_i/I^{n+1}F_i \rt M_i/I^{n+1}M_i \rt 0.
\]
Similarly we have an exact sequence
\[
0 \rt \Tor^A_{i+1}(M, A/I^{n}) \rt M_{i+1}/I^{n}M_{i+1} \rt F_i/I^{n}F_i \rt M_i/I^{n}M_i \rt 0.
\]
Computing lengths we have
\begin{align*}
\ell\left(\frac{I^nM_i}{I^{n+1}M_i}\right) + \ell\left(\frac{I^nM_{i+1}}{I^{n+1}M_{i+1}}\right)  &= (\rank{F_i})\ell\left(\frac{I^n}{I^{n+1}}\right) + \ell\left(\Tor^A_{i+1}(M, A/I^{n+1}) \right)  \\
&\  \   \  - \ell \left( \Tor^A_{i+1}(M, A/I^n) \right).
\end{align*}
By \ref{growth} the terms in the right hand side are of quasi-polynomial type with period $2$ and degree $\leq \cx(M)-1$. The result follows from 
Corollary \ref{g2-c}.
\end{proof}
As a consequence of the above Proposition we get:
\begin{corollary}\label{zero-dim}
Let $(A,\m)$ be an Artin local ring and let $I$ be an $\m$-primary ideal. Let $M$ be a finitely generated $A$-module of finite GCI-dimension. Then for all $j \geq 0$ the function
$i \mapsto e_j^I(\Syz^A_{i}(M))$ is of quasi-polynomial type with period $2$ and degree $\leq \cx(M)-1$.
\end{corollary}
\begin{proof}
We note that $I^r = 0$ for some $r \geq 1$. The result now follows from Corollary \ref{hilbSyz}.
\end{proof}

\s\label{Eis-reg} \emph{Eisenbud operators modulo a regular element.} \\
Let $x $ be $A \oplus M$-regular. Let $y \in Q$ be a pre-image of $x$ in $Q$. As permutation of regular sequences is also a regular sequence  we also get that $y$ is $Q$-regular.
We now note that $\projdim_{Q/(y)}(M/xM)$ is finite. Let $\mathbb{F}$ be a minimal resolution of $M$ over $A$. Then $\mathbb{F}/x\mathbb{F}$ is a minimal resolution of $M/xM$ over $A/(x)$. Let $t_1,\ldots, t_c$ be the Eisenbud operators  over $\mathbb{F}$.  By the construction of Eisenbud operators it follows that $t_1,\ldots,t_c$ induce operators $t_1^*,\ldots, t_c^*$ over $\mathbb{F}/x\mathbb{F}$.  Let $N$ be an $A/(x)$-module.  Set $E =  \Tor^A_*(M, N) = \Tor^{A/(x)}_*(M/xM, N)$. Then the action of $t_i$ on $E$ is same as that of $t_i^*$. 

\s  \label{bigraded} Now assume that $Q$ be a Noetherian ring  and let $\mathbf{f}= f_1,\ldots f_c$ be a regular sequence in $Q$.  Set $A = Q/(\mathbf{f})$.  Let $M$ be a finitely generated $A$-module with $\projdim_Q M$ finite. 
Let $\mathbb{F}$ be a minimal resolution of $M$.  Let $t_1,\cdots, t_c \colon \mathbb{F}(+2) \rt \mathbb{F}$ be the Eisenbud operators.
Let  $I$ be an $\m$-primary ideal of $A$ and let $\R = A[Iu]$ be the Rees algebra of $A$ \wrt \ $I$.  Let $N = \bigoplus_{n \geq 0} N_n$ be an $\R$-module (not necessarily finitely generated). We claim that $\bigoplus_{i \geq 0} \Tor^A_i(M, N) = \bigoplus_{i, n \geq 0} \Tor^A_i(M, N_n)$ has a bigraded $\R[t_1,\ldots,t_c]$-module structure. Here $t_i $ have degree $(0,2)$ and if $a \in I^n$ then $\deg au^n = (n, 0)$.  To see this let $v = au^s \in \R_s$. 
The map $N_n \xrightarrow{v} N_{n+s}$
yields  the following commutative diagram
\[
\xymatrix
{
\mathbb{F}\otimes N_n \ar@{->}[d]^{v}\ar@{->}[r]_{t_j}
&\mathbb{F}\otimes N_n(-2) \ar@{->}[d]^{v}
\\
\mathbb{F}\otimes N_{n+s} \ar@{->}[r]_{t_j}
&\mathbb{F}\otimes N_{n+s}(-2).
}
\]

Taking homology we get the required result. 
An analogous argument yields that  $\bigoplus_{i ,n\geq 0} \Ext^i_A(M, N_n)$ is a bigraded $\R[t_1,\ldots, t_c]$-module.

\section{$L^I_i(M)$}
In this section we extend and simplify a technique from \cite{PuMCM} and \cite{Pu2}.
Let $A$ be a Noetherian ring, $I$ an $\m$-primary ideal and $M$ a finitely generated 
$A$-module.

\s Set $L^I_0(M) = \bigoplus_{n \geq 0} M/I^{n+1}M$. Let $\R = A[Iu]$ be the \emph{Rees-algebra} of $I$. Let $\Sc = A[u]$. Then $\R$ is a subring of $\Sc$. Set
$M[u] = M\otimes_A \Sc$ an $\Sc$-module and so an $\R$-module. Let $\R(M) = \bigoplus_{n \geq 0}I^nM$ be the Rees-module of $M$ with respect to $I$. We have the following exact sequence of $\R$-modules
\[
0 \rt \R(M) \rt M[u] \rt L^I_0(M)(-1) \rt 0.
\]
Thus $L^I_0(M)(-1)$ (and so $L^I_0(M)$) is a $\R$-module. We note that $L^I_0(M)$ is \emph{not} a finitely generated $\R$-module. Also note that $L^I_0(M) = M \otimes_A L^I_0(A)$.

\s For $i \geq 1$ set 
$$L^I_i(M) = \Tor^A_i(M, L^I_0(A)) = \bigoplus_{n \geq 0 } \Tor^A_i(M, A/I^{n+1}). $$
We assert that $L^I_i(M)$ is a finitely generated $\R$-module for $i \geq 1$. It is sufficient to prove it for $i = 1$. We tensor the exact sequence $0 \rt \R \rt \Sc \rt L^I_0(A)(-1) \rt 0$ with $M$ to obtain a sequence of $\R$-modules
\[
0 \rt L^I_1(M)(-1) \rt \R \otimes_A M \rt M[u] \rt  L^I_0(M)(-1) \rt 0.
\] 
Thus $ L^I_1(M)(-1)$ is a $\R$-submodule of $\R \otimes_A M$. The latter module is a finitely generated $\R$-module. It follows that $L^I_1(M)$ is a finitely generated $\R$-module.

\s Now assume that $A$ is \CM \ and $M$ is maximal \CM.  Set $N = \Syz^A_1(M)$ and $F = A^{\mu(M)}$ (here $\mu(M)$ is the cardinality of  a minimal generator set  of $M$). We tensor the exact sequence
\[
0 \rt N \rt F  \rt M \rt 0,
\]
with $L^I_0(A)$ to obtain an exact sequence of $\R$-modules
\[
0 \rt L^I_1(M) \rt  L^I_0(N) \rt  L^I_0(F) \rt  L^I_0(M) \rt 0.
\]
As $e_0^I(F) = e_0^I(M) + e_0^I(N)$ we get that the function $n \rt \ell(\Tor^A_1(M, A/I^{n+1}))$ is of  polynomial type and degree $\leq d - 1$. Thus $\dim L^I_1(M) \leq d$.

We need the following result.
\begin{proposition}\label{H1}
Let $(A,\m)$ be \CM \ local ring of dimension $d \geq 1$ with infinite residue field and let $I$ be an $\m$-primary ideal.  Let $M$ be a MCM  \  $A$-module. Set $M_1 = \Syz^A_1(M)$. Let $x$ be $A \oplus M \oplus M_1$-superficial \wrt \ $I$.  Let $\R = A[Iu]$ be the Rees algebra of $A$ \wrt  \  $ I$. Set $X = xu \in \R_1$.  If $N$ is a $\R$-module, let $H_1(X, N)$ denote the first Koszul homology of $N$ \wrt \  $X$. Then
\begin{enumerate}[\rm (1)]
\item
$H_1(X, L_0^I(M))  =  \bigoplus_{n \geq 0}(I^{n+1}M \colon x)/I^nM$.
\item
$X$ is $L_0^I(M)$ regular if and only if $x^*$ is $G_I(M)$-regular.
\item
If $x^*$ is $G_I(A)$-regular then 
$H_1(X,L_1^I(M))  \cong H_1(X, L_0^I(M_1))$.
\end{enumerate} 
\end{proposition}
\begin{proof}
(1) The map $M/I^nM \xrightarrow{X} M/I^{n+1}M$ is given by $m + I^nM \mapsto xm + I^{n+1}M$. The result follows.

(2) This follows from (1).

(3) We have an exact sequence $0 \rt L_1^I(M) \rt L_0^I(M_1) \rt L_0^I(F_0) \xrightarrow{\pi} L_0^I(M) \rt 0$. Let $E = \ker \pi$. As $x^*$ is $G_I(A)$-regular we get that $H_1(X, L_0^I(A)) = 0$. So $H_1(X, E) = 0$.  The result follows. 
\end{proof}

\textbf{II}:  \textit{Ratliff-Rush filtration.} 

Let $(A,\m)$ be a Noetherian  local ring and let $I$ be an $\m$-primary  ideal in $A$. Let $M$ be a finitely generated $A$-module. The Ratliff-Rush submodule $\widetilde{IM}$ of $M$ \wrt \ $I$
is defined by
\[
\widetilde{IM} = \bigcup_{k \geq 0} (I^{k+1}M \colon I^k).
\]
If $\depth M > 0 $ then $\widetilde{I^kM} =  I^kM $ for $k \gg 0$. 

\s Now assume $d = \dim A > 0$ and $A$ is \CM. Also assume $M$ is MCM $A$-module. Let $\R$ be the Rees algebra of $A$ \wrt \ $I$ and let $\mathcal{M}$ be its unique maximal graded ideal. Then by \cite[4.7]{Pu2} we have
\[
H^0_\mathcal{M}(L^I_0(M)) = \bigoplus_{n \geq 0} \frac{ \widetilde{I^{n+1}M}}{I^{n+1}M}.
\]
We need the following result:
\begin{proposition}\label{RR}[with assumptions as above]
If $\depth G_I(A) > 0$ then 
$$H^0_\mathcal{M}(L^I_1(M)) \cong H^0_\mathcal{M}(L^I_0(\Syz^A_1(M))). $$
\end{proposition}
\begin{proof}
Set $M_1 = \Syz^A_1(M)$. The exact sequence $0 \rt M_1 \rt F \rt M \rt 0$ yields an exact sequence
\[
0 \rt L_1^I(M) \rt L^I_0(M_1) \rt L_0^I(F) \rt L^I_0(M) \rt 0.
\]
Therefore we have an exact sequence
\[
0 \rt H^0_\M(L^I_1(M)) \rt H^0_\M(L^I_0(M_1)) \rt H^0_\M(L_0^I(F))
\]
Now as $\depth G_I(A) > 0$ we get that $\widetilde{I^n} = I^n $ for all $n \geq 1$. So $H^0_\M(L^I_0(A)) = 0$. Thus $H^0_\M(L^I_0(F)) = 0$. The result follows.
\end{proof}

\s \label{bi} Now assume that $A = Q/(\mathbf{f})$ where $\mathbf{f} = f_1,\ldots, f_c$ is a $Q$-regular sequence and $M$ is a finitely generated $A$-module with $\projdim_Q M$ finite. Let
$I$ be an $\m$-primary ideal. Set $L(M) = \bigoplus_{i \geq 0}L_i^I(M) = \bigoplus_{i \geq 0}\Tor^A_i(M, L^I_0(A))$.  Give $L(M)$ a structure of a bigraded $S = \R[t_1,\cdots,t_c]$-module as discussed in \ref{bigraded}.
Then $H^0_\M(L(M))  = \bigoplus_{i\geq 0}H^0_\M(L^I_i(M))$ is a bigraded $S$-module. To see this note that $H^0_\M(L(M)) = H^0_{\M S}(L(M))$. 
\section{dimension one}
In this section we assume $\dim A = 1$. We prove that if $M$ is MCM $A$-module with finite GCI dimension then the function $i \mapsto e_1^I(\Syz^A_i(M))$ is of quasi-polynomial type with period two. If $G_I(A)$ is \CM \ then we show that the functions $j \mapsto \depth G_I(\Syz^A_{2j}(M))$ and $j \mapsto \depth G_I(\Syz^A_{2j+1}(M))$ is constant for $j \gg 0$.

The following Lemma is crucial.
\begin{lemma}\label{crucial}
Let $(A,\m)$ be a \CM \ local ring of dimension one and with an infinite residue field. Let $I$ be an $\m$-primary ideal and let $x$ be $A$-superficial with respect to $I$. Assume $I^{r+1} = xI^r$. Let $M$ be a MCM $A$-module. Then
\begin{enumerate}[\rm (1)]
\item
$\deg h^I(M,z)  \leq r$
\item
The postulation number of the Hilbert-Samuel function of $M$ \wrt \ $I$ is $\leq r - 2$.
\item
For $n \geq r$ and for all $i \geq 1$ we have
\[
\Tor^A_i(M, A/I^{n}) \cong \Tor^A_i(M, A/I^{n+1}) \quad \text{as $A$-modules}.
\]
\item
$\widetilde{I^nM} = I^nM$ for all $n \geq r$.
\end{enumerate}
\end{lemma}
\begin{proof}
We first note that as $x$ is $A$-regular it is also $M$-regular. Thus the equality $I^{r+1}M = xI^rM$ implies that $(I^{n+1}M \colon x) = I^nM$ for $n \geq r$. Therefore $x$ is $M$-superficial \wrt \ $I$. 

(1) This follows from \cite[Proposition 13]{Pu1}.

(2) This follows from (1).

(3) Set $B  = A/(x)$. For $n \geq r $ we have an exact sequence
\[
0 \rt A/I^n \xrightarrow{\alpha_n} A/I^{n+1}  \rt B \rt 0.
\]
Here $\alpha_n( a + I^n) = xa + I^{n+1}$. We have also used that for $n \geq r $,  $I^{n+1} \subseteq (x)$. As $\Tor^A_i(M, B) = 0$ for $i \geq 1$ we get the required result.

(4) As $(I^{n+1}M \colon x) = I^n M$ for all $n \geq r$, by  \cite[2.7]{Pu2} we get $\widetilde{I^nM} = I^nM$ for $n \geq r$.
\end{proof}

As an immediate corollary we obtain
\begin{theorem}\label{e1}
Let $(A,\m)$ be \CM \ local ring  of dimension $d$ and let $I$ be an $\m$-primary ideal. Let $M$ be a MCM $A$-module. Assume $M$ has finite GCI-dimension over $A$.
Then the function $i \mapsto e_1^I(\Syz^A_i(M))$ is of  quasi-polynomial type with period two and degree $\leq \cx(M)-1$.
\end{theorem}
\begin{proof}
By  Corollary \ref{zero-dim} the result holds when $A$ is Artin.
By Remark \ref{base-change} we may assume $A = Q/(\mathbf{f})$ where $Q$ is complete with uncountable residue field, $\mathbf{f} = f_1,\cdots, f_c$ is a $Q$-regular sequence and 
$\projdim_Q M$ is finite.  Let $\mathbb{F}$ be a minimal resolution of $M$. Set $M_i = \Syz^A_i(M)$. If $d \geq 2$ then by  \ref{count-sup} we may choose $\mathbf{x} = x_1,\ldots,x_{d-1}$ which is $A \oplus M_i$-superficial for all $i \geq 0$. As $e^I_1(M_i) = e_1^I(M_i/ \mathbf{x}M_i)$ for all $i \geq 0$, we may assume $\dim A = 1$.

Let $x$ be $A$-superficial \wrt \ $I$. Say $I^{r+1} = xI^r$. Then by
Lemma \ref{crucial} $x$ is $M_i$-superficial for all $i \geq 0$ and the postulation number of  the Hilbert-Samuel function of $M_i$ \wrt  \ $I$ is $\leq r - 2$.

  Fix $n \geq r$. The exact sequence
$0 \rt M_{i+1} \rt F_i \rt M_i \rt 0$ yields an exact sequence
\[
0 \rt \Tor^A_{i+1}(M, A/I^{n+1}) \rt M_{i+1}/I^{n+1}M_{i+1} \rt F_i/I^{n+1}F_i \rt M_i/I^{n+1}M_i \rt 0.
\]
Using Lemma \ref{crucial},
it follows that
\[
\ell(\Tor^A_{i+1}(M, A/I^{n+1})) = (\rank F_i)e_1^I(A) - e_1^I(M_i) - e_1^I(M_{i+1}).
\]
The functions $i \mapsto \rank F_i$ and  $i \mapsto \ell(\Tor^A_{i}(M, A/I^n))$ are of quasi-polynomial type with period two and degree $\leq \cx(M)-1$. The result now follows from \ref{g2-c}.
\end{proof}

We now prove:

\begin{theorem}\label{asymp-depth-1}
Let $(A,\m)$ be \CM \ local ring of dimension one and let $I$ be an $\m$-primary ideal with $G_I(A)$ \CM.  Let $M$ be a MCM $A$-module. Assume $M$ has finite GCI dimension over $A$. Then the functions $j \mapsto \depth G_{I}(\Syz^A_{2j}(M))$ and $j \mapsto \depth G_I(\Syz^A_{2j+1}(M))$ are constant for $j \gg 0$
\end{theorem}
\begin{proof}
By Remark \ref{base-change} we may assume $A = Q/(\mathbf{f})$ where $Q$ is complete with infinite residue field, $\mathbf{f} = f_1,\cdots, f_c$ is a $Q$-regular sequence and 
$\projdim_Q M$ is finite.  Let $\mathbb{F}$ be a minimal resolution of $M$. Set $M_i = \Syz^A_i(M)$.

Let $x$ be $I$-superficial. Then $x^*$ is $G_I(A)$-regular. Say $I^{r+1} = xI^r$. Then by Lemma \ref{crucial}  we get that $x$ is $M_i$-superficial for all $i \geq 0$. We have
$(I^{n+1}M_i \colon x)  = I^nM_i$ for all $n \geq r$ and for all $i \geq 0$.

Let $t_1,\ldots,t_c$ be the Eisenbud operators on $\mathbb{F}$.  Let $\R$ be the Rees algebra of $A$ \wrt \  $I$. By \ref{bigraded} we get that $E  = \bigoplus_{i \geq 0} L_i^I(M)$ is a bigraded $S =  \R[t_1,\ldots, t_c]$-module.
Set $X = xu$. Then $H_1(X,E)$ is a bigraded $S$-module. By \ref{H1}  we get that $H_1(X,E) = \bigoplus_{i \geq 0} H_1(X, L_0^I(M_i))$.  We now note that $H_1(X,E) \subseteq E_{n \leq r-1}$ and so is  an Artinian $A[t_1,\ldots, t_c]$-module. It follows that the function $j \mapsto \ell(H_1(X, L_0^I(M_i)))$ is of  quasi-polynomial type with period two.  By \ref{H1} the result follows.
\end{proof}

We now prove
\begin{proposition}\label{e2-dim1}
Let $(A,\m)$ be \CM \ local ring of dimension one and let $I$ be an $\m$-primary ideal. Let $M$ be a MCM $A$-module of finite GCI dimension over $A$.
Then the function $i \mapsto e_2^I(\Syz^A_i(M))$ is of quasi-polynomial type with period two and degree $\leq \cx(M)-1$.
\end{proposition}
\begin{proof}
By Remark \ref{base-change} we may assume $A = Q/(\mathbf{f})$ where $Q$ is complete with infinite residue field, $\mathbf{f} = f_1,\cdots, f_c$ is a $Q$-regular sequence and 
$\projdim_Q M$ is finite.  Let $\mathbb{F}$ be a minimal resolution of $M$. Set $M_i = \Syz^A_i(M)$.

Let $x$ be $I$-superficial. Say $I^{r+1} = xI^r$. Then by Lemma \ref{crucial}  we get that 
\begin{equation}\label{eqn-1}
\Tor^A_i(M, A/I^n) \cong \Tor^A_i(M, A/I^{n+1}) \quad \text{ for all $n \geq r$.} 
\end{equation}
Fix $n \geq 0$. Then by \ref{growth} the function $i \mapsto \ell(\Tor^A_i(M, A/I^{n+1}))$ is quasi-polynomial of period two and degree $\leq \cx(M)-1$.  Thus by (\ref{eqn-1}) we get that there exists $p(z,w) \in \mathbb{Z}[z,w]$ such that
\[
\sum_{i \geq 1, n \geq 0} \ell(\Tor^A_i(M, A/I^{n+1}))z^iw^n  = \frac{p(z,w)}{(1-z^2)^{\cx(M)}(1-w)}. 
\]
Set  
$$f(i,n) =  \sum_{m = 0}^{n}\ell(\Tor^A_i(M, A/I^{m+1})).$$ 
Multiplying the previous equation 
by $1/(1-w)$ we get that
\[
\sum_{i \geq 1, n \geq 0}  f(i,n)z^iw^n  = \frac{p(z,w)}{(1-z^2)^{\cx(M)}(1-w)^2}.
\]
By  \ref{g2-growth} we get that
for $i,n \gg 0$
$$ f(i,n) = g_0(i)(n+1) - g_1(i) $$
where for $j = 0, 1$ the function  $i \mapsto g_j(i)$ is of quasi-polynomial type with period $2$ and degree $ \leq \cx(M) - 1$. 
It remains to note that
\begin{align*}
g_0(i)  &=  (\rank  F_i)e_1^I(A) - e_1^I(M_i) - e_1^I(M_{i-1}),  \text{and} \\
g_1(i) &= (\rank F_i)e_2^I(A) - e_2^I(M_i) - e_2^I(M_{i-1}).
\end{align*}
The result now follows from Corollary \ref{g2-c}.
\end{proof}
We need the following result in the next section.
\begin{proposition}\label{RR-artin-dim1}
Let $(A,\m)$ be a one dimensional \CM \ local ring  with infinite residue field and let $I$ be an $\m$-primary ideal of $A$.  Let $M$ be a maximal \CM \ $A$-module. Assume $A  = Q/(\mathbf{f})$  where $\mathbf{f} = f_1,\ldots, f_c$ is a $Q$-regular sequence and $\projdim_Q M $ is finite. Let $\mathbb{F}$ be a minimal resolution of $M$ over $A$. Let $t_1,\ldots,t_c$ be the Eisenbud operators over $\mathbb{F}$. Let $\R = A[Iu]$ be the Rees algebra of $A$ \wrt \ $I$ and let $\M$ be its unique maximal homogeneous ideal.
Then $\bigoplus_{i \geq 0} H^0_\M(L^I_i(M))$ is an Artinian $A[t_1,\ldots,t_c]$-module. Furthermore the function $i \mapsto \ell(H^0_\M(L^I_i(M)))$ is of quasi-polynomial type with period two and  degree $\leq \cx(M)-1$.
\end{proposition}
\begin{proof}
Let $x$ be $A$-superficial \wrt \ $I$. Say $I^{r+1} = xI^r$.  Set $M_i = \Syz^A_i(M)$ for $i \geq 0$. Then by Lemma \ref{crucial}.4 we get that $\widetilde{I^n M_i} = I^nM_i$ for $n \geq r$. So $H^0_\M(L^I_0(M_i))_n =0$ for $n \geq r-1$. The exact sequence $0 \rt M_{i+1} \rt F_i \rt M_i \rt 0$ yields an exact sequence
\[
0 \rt L_{i+1}^I(M) \rt L_0^I(M_{i+1}) \rt L_0^I(F_i) \rt L_0^I(M_i) \rt 0.
\]
So we have an inclusion $0 \rt H^0_\M(L^I_{i+1}(M)) \rt H^0_\M(L_0^I(M_{i+1}))$. It follows that $H^0_\M(L^I_{i+1}(M))_n = 0$ for $n \geq r-1$.

By \ref{bigraded},   $L = \bigoplus_{i \geq 0}L^I_i(M)$ is a bigraded $S = \R[t_1,\ldots,t_c]$-module. Then  \\ $H^0_\M(L) = \bigoplus_{i \geq 0} H^0_\M(L^I_i(M))$ is also a bigraded $S$-module. As $H^0_\M(L) \subseteq E = L_{n \leq r -1}$ it follows that  $H^0_\M(L)$ is an Artinian $A[t_1,\ldots,t_c]$-module. It follows that the function $i \mapsto \ell(H^0_\M(L^I_i(M)))$ is of quasi-polynomial type with period two. To compute its degree note that the function $i \mapsto \ell(E_{i})$  has degree $\leq \cx(M)-1$. 

\end{proof}
\section{dimension 2}
In this section we assume $(A,\m)$ is \CM \ with  $\dim A = 2$ . Let $I$ be an $\m$-primary ideal. If $M$ is a MCM $A$-module with finite GCI dimension over $A$ then we prove that the function $i \mapsto e_2^I(\Syz^A_i(M))$ is of quasi-polynomial type with period two.  If $G_I(A)$ is \CM \ then we show that the functions  $j \mapsto \depth G_I(\Syz^A_{2j}(M))$ and $j \mapsto \depth G_I(\Syz^A_{2j+1}(M))$ are constant for $j \gg 0$.

\begin{theorem}\label{e2}
Let $(A,\m)$ be a \CM \ local ring of dimension $d$ and let $I$ be an $\m$-primary ideal. Let $M$ be a MCM $A$-module. Assume that $M$ has finite GCI dimension over $A$.
Then the function $i \mapsto e_2^I(\Syz^A_i(M))$ is of quasi-polynomial type with period two and degree $\leq \cx(M)-1$.
\end{theorem}
\begin{proof}
By  Corollary \ref{zero-dim} the result holds when $A$ is Artin. By Proposition \ref{e2-dim1} the result holds when $d  = 1$.
By Remark \ref{base-change} we may assume $A = Q/(\mathbf{f})$ where $Q$ is complete with uncountable residue field, $\mathbf{f} = f_1,\cdots, f_c$ is a $Q$-regular sequence and 
$\projdim_Q M$ is finite.  Let $\mathbb{F}$ be a minimal resolution of $M$. Set $M_i = \Syz^A_i(M)$. If $d \geq 3$ then by  \ref{count-sup} we may choose $\mathbf{x} = x_1,\ldots,x_{d-2}$ which is $A \oplus M_i$-superficial for all $i \geq 0$. As $e_2^I(M_i) = e_2^I(M_i/ \mathbf{x}M_i)$ for all $i \geq 0$, we may assume $\dim A = 2$.

We note that for any $n \geq 1$ we have $e_2^{I^n}(M) = e_2^I(M)$. For $n \gg 1$  we have $\depth G_{I^n}(A) \geq 1$. Thus we can assume $\depth G_I(A) > 0$.
Let $x$ be $A\oplus M_i$-superficial for all $i \geq 0$. Then $x^*$ is $G_I(A)$-regular. Let $\R = A[Iu]$ be the Rees algebra of $A$ \wrt \ $I$. Set $X = xu \in \R_1$ and $\ov{A} = 
A/(x)$.  Also set $\ov{M_i}  =  M_i/xM_i$. 

Let $t_1,\ldots, t_c$ be the Eisenbud operators over $\mathbb{F}$. Then $L(M) = \bigoplus_{i \geq 0} L^I_i(M)$ is a bigraded $S = \R[t_1,\ldots, t_c]$-module (see \ref{bi}).
Notice $X$ is $L^I_0(A)$-regular.  So we have an exact sequence of $\R$-modules
\[
0 \rt L^I_0(A)(-1) \xrightarrow{X} L^I_0(A) \rt L^I_0(\ov{A}) \rt 0.
\]
This induces an exact sequence of $S$-modules
\begin{equation}\label{two}
0 \rt K \rt L(M)(-1,0) \xrightarrow{X} L(M) \rt L(\ov{M})
\end{equation}
By Proposition \ref{H1} we get that
\begin{equation}\label{2.5}
K =  \bigoplus_{i \geq 0} H_1(X, L^I_0(M_i)).
\end{equation}
Let $\M$ be the unique maximal homogeneous maximal ideal of $\R$. We take local cohomology \wrt \ $\M$ (on (\ref{two})).
As $K$ is $\M$-torsion we get an exact sequence  of $S$-modules
\begin{equation}\label{three}
0 \rt K \rt H^0_\M(L(M))(-1,0) \rt H^0_\M(L(M)) \xrightarrow{\rho} H^0_\M(L(\ov{M})).
\end{equation}
By \ref{Eis-reg} and \ref{RR-artin-dim1} we get that $H^0_\M(L(\ov{M}))$ is an Artin
module over the subring 
$T = A[t_1,\cdots, t_c]$ of $S$. So $E = \image \rho$ is also an Artin $T$-module.
By \ref{RR} we get that
\[
H^0_\M(L(M))  = \bigoplus_{i \geq 0}H^0_\M(L^I_0(M_i)).
\] 
Thus for a fixed $i$ we get that $H^0_\M(L(M))_{(i,n)} = 0$ for $n \gg 0$.
It follows that $\ell(K_i) = \ell(E_i)$. As $E$ is an Artin $T$-module the function 
$i \mapsto \ell(E_i)$ is of quasi-polynomial type with period two. By \ref{RR-artin-dim1} its degree is $\leq \cx(M) - 1$. Thus the  function 
$i \mapsto \ell(K_i)$ is of quasi-polynomial type with period two and degree $ \leq \cx(M)-1$. By \ref{mod-sup-h}(3b)
\ref{H1}(1)  and (\ref{2.5})  we have
\[
e_2^I(M_i) = e_2^I(\ov{M_i}) - \ell(K_i),
\]
is of quasi-polynomial type with period two and degree $\leq \cx(M)-1$.
\end{proof}
Next we prove:
\begin{theorem}\label{asymp-depth-2}
Let $(A,\m)$ be \CM \ local ring of dimension two and let $I$ be an $\m$-primary ideal with $G_I(A)$ \CM.  Let $M$ be a MCM $A$-module. Assume $M$ has finite GCI dimension over $A$. Then the functions $j \mapsto \depth G_{I}(\Syz^A_{2j}(M))$ and $j \mapsto \depth G_I(\Syz^A_{2j+1}(M))$ are constant for $j \gg 0$
\end{theorem}
To prove this theorem we need the following:
\begin{lemma}\label{depth1}
Let $(A,\m)$ be a \CM \ local ring of dimension two and let $I$ be an $\m$-primary ideal. Let $M$ be an MCM $A$-module. Let $x,y$ be an $A\oplus M$-superficial sequence. Set $J = (x,y)$ and assume $I^{r+1} = JI^r$. Then the following are equivalent:
\begin{enumerate}[\rm (1)]
\item
$x^*$ is $G_I(M)$-regular.
\item
$(I^{n+1}M \colon x) = I^nM$ for all $n \geq 1$.
\item
$(I^{n+1}M \colon x) = I^nM$ for $n = 1,\cdots, r$.
\end{enumerate}
\end{lemma}
\begin{proof}
The equivalence (1) and (2) is well-known. Also clearly (2) implies (3).
Now assume (3). Set $N = M/xM$. We have the following exact sequence
\[
 0 \rt  \frac{(I^{n}M \colon J)}{I^{n-1}M}     \xrightarrow{\gamma_n}                     \frac{(I^{n}M \colon x)}{I^{n-1}M}      \xrightarrow{\beta_n}         \frac{(I^{n+1}M \colon x)}{I^nM}      \xrightarrow{\alpha_n}    \frac{I^{n+1}M}{JI^nM}     \xrightarrow{\rho_n}  \frac{I^{n+1}N}{yI^nN}\rt 0,
\]
where $\rho_n$ is the natural surjection, $\gamma_n$ is the natural inclusion, $\alpha_n( a + I^nM) = xa + JI^nM$ and $\beta_n(a + I^{n-1}M) = ya + I^nM$.
Set
\[
U = \bigoplus_{n \geq 0}\frac{(I^{n}M \colon J)}{I^{n-1}M} \quad \text{and} \quad V =  \bigoplus_{n \geq 0}\frac{(I^{n+1}M \colon x)}{I^nM}.  
\]
If $(I^{n+1}M \colon x) = I^nM$ for $n = 1,\cdots, r$ then by the above exact sequence we have an exact sequence
\[
0 \rt U \rt V[-1] \rt V \rt 0
\]
As $\ell(V)$ is finite we get that $V[-1] \cong V$ and this implies $V = 0$, so (2) holds.
\end{proof}
We now give
\begin{proof}[Proof of Theorem \ref{asymp-depth-2}]
By Remark \ref{base-change} we may assume $A = Q/(\mathbf{f})$ where $Q$ is complete with uncountable  residue field, $\mathbf{f} = f_1,\cdots, f_c$ is a $Q$-regular sequence and 
$\projdim_Q M$ is finite.  Let $\mathbb{F}$ be a minimal resolution of $M$. Set $M_i = \Syz^A_i(M)$.

Let $x,y$ be $M_i\oplus A$-superficial sequence \wrt \ $I$ for all $i$. Such an element exists as the residue field of $A$ is uncountable, see \ref{count-sup}. Then $x^*$ is $G_I(A)$-regular.
Let $\R = A[Iu]$ be the Rees algebra of $A$ \wrt \ $I$. Set $X = xu \in \R_1$ and $\ov{A} = 
A/(x)$.  Also set $\ov{M_i}  =  M_i/xM_i$. 

Let $t_1,\ldots, t_c$ be the Eisenbud operators over $\mathbb{F}$. Then $L(M) = \bigoplus_{i \geq 0} L^I_i(M)$ is a bigraded $S = \R[t_1,\ldots, t_c]$-module (see \ref{bi}).

Notice $X$ is $L^I_0(A)$-regular.  So we have an exact sequence of $\R$-modules
\[
0 \rt L^I_0(A)(-1) \xrightarrow{X} L^I_0(A) \rt L^I_0(\ov{A}) \rt 0.
\]
This induces an exact sequence of $S$-modules
\begin{equation}\label{four}
0 \rt K \rt L(M)(-1,0) \xrightarrow{X} L(M) \rt L(\ov{M}).
\end{equation}
By Proposition \ref{H1} we get that
\begin{equation}\label{five}
K =  \bigoplus_{i \geq 0} H_1(X, L^I_0(M_i)).
\end{equation}
Let $I^{r+1} = (x,y)I^r$. We note that $E = K_{n \leq r}$ is an Artin $A[t_1,\ldots, t_c]$-module. Thus the function $i \rt \ell(E_i)$ is of quasi-polynomial type with period two. Therefore by Lemma \ref{depth1} we  get  that for each $j = 0, 1$  either  $\depth G_I(M_{2i+j}) = 0 $ for $i  \gg 0$ or $\depth G_I(M_{2i+j}) \geq 1 $ for $i  \gg 0$. We now go mod $x$ and use Theorem \ref{asymp-depth-1} and \ref{mod-sup-h}(5) to conclude.
\end{proof}
\section{asymptotic depth}
In this section $(A,\m)$ is a \CM \ local ring of dimension two and $I$ is an $\m$-primary ideal with $G_I(A)$ \CM.  Let $M$ be a MCM $A$-module with finite GCI-dimension over $A$. Let $\xi_I(M) = \lim_{n \rt \infty} \depth G_{I^n}(M)$. In this section we prove that the functions $i \mapsto \xi_I(\Syz^A_{2i}(M))$ and $i \mapsto \xi_I(\Syz^A_{2i+1}(M))$ are constant for  $i \gg 0$.   
Let $\R(I)$ denote the Rees algebra of $A$ \wrt \ $I$ and let $\M$ denote its unique maximal homogeneous ideal.

\s \label{reductions}  We note that $\xi_I(M) \geq 1$. Furthermore $\xi_I(M) \geq 2$ if and only if \\  $H^1_\M(L^I_0(M))_{-1} = 0$, see \cite[9.2]{Pu2}. As $\dim A = 2$ the only possible values of $\xi_I(M)$ is $1$ or $2$. We now note that $L^I_0(M)(-1)  = \bigoplus_{n \geq 0} M/I^nM$ behaves very well \wrt \ to the Veronese functor. Clearly for $m \geq 1$ we have
\[
L^I_0(M)(-1)^{<m>}  = \left( \bigoplus_{n \geq 0} M/I^nM \right)^{<m>} = \bigoplus_{n \geq 0}M/I^{nm} = L^{I^m}_{0}(M)(-1).
\]
Also note that $\M^{<m>}$ is the unique maximal homogeneous ideal of $\R(I^m)$. It follows that
\[
H^1_\M(L^I_0(M))_{-1}  \cong  H^1_{\M^{<m>}}(L^{I^m}_0(M))_{-1} \quad \text{as $A$-modules}.
\]
We need the following:
\begin{lemma}\label{need}
Let $(A,\m)$ be a \CM \ local ring of dimension two and let $I$ be an $\m$-primary ideal. Let $J = (x,y)$ be a minimal reduction of $I$ and assume that  $I^{r+1} = JI^r$. Let
$M$ be a maximal \CM \ $A$-module. Then $H^1_\M(L^I_0(M))_n = 0$ for $n \geq r - 2$. In particular if $m \geq r$ then
\[
H^1_{\M^{<m>}}(L^{I^m}_0(M))_n = 0 \quad \text{for} \ n \geq 0.
\]
\end{lemma}
\begin{proof}
Let $a_2(G_I(M))  = \max \{ n \mid H^2_{\M}(G_I(M))_n \neq 0 \}$. Then by 
\cite[3.2]{T} we get that $a_2(G_I(M))  \leq r - 2$. The exact sequence
$0 \rt G_I(M) \rt L^I_0(M) \rt L^I_0(M)(-1) \rt 0$ induces an exact sequence
\[
H^1_\M(L^I_0(M)) \rt H^1_\M(L^I_0(M))(-1) \rt H^2_\M(G_I(M)). 
\]
By \cite[6.4]{Pu2} we get that $H^1_\M(L^I_0(M))_n = 0$ for $n \gg 0$. From the above exact sequence it follows that
$H^1_\M(L^I_0(M))_n = 0 $ for all $n \geq r - 2$. The rest of the assertion follows from \ref{reductions}.
\end{proof}
We now state and prove the main result of this section.
\begin{theorem}\label{asymp}
Let $(A,\m)$ be a \CM \ local ring of dimension two and let $I$ be an $\m$-primary ideal with $G_I(A)$ \CM. Let $M$ be a MCM $A$-module and assume $M$ has finite GCI dimension over $A$.  Then
the functions $i \mapsto \xi_I(\Syz^A_{2i}(M))$ and $i \mapsto \xi_I(\Syz^A_{2i+1}(M))$ are constant for  $i \gg 0$.   
\end{theorem}
\begin{proof}
By Remark \ref{base-change} we may assume $A = Q/(\mathbf{f})$ where $Q$ is complete with uncountable residue field, $\mathbf{f} = f_1,\cdots, f_c$ is a $Q$-regular sequence and 
$\projdim_Q M$ is finite.  Let $\mathbb{F}$ be a minimal resolution of $M$. Set $M_i = \Syz^A_i(M)$. Let $\R(I) = A[Iu]$ denote the Rees algebra of $A$ \wrt \ $I$ and let $\M$ denote its unique maximal homogeneous ideal. By Lemma \ref{need} and \ref{reductions} we may assume that $H^1_\M(L^I_0(M_i))_n = 0 $ for $n \geq 0$ and for all $i \geq 0$.
We will show that the vanishing of $H^1_\M(L^I_0(M_i))_{-1}$ can be detected by a quasi-polynomial of period two.

\noindent\textit{Claim-1:} $H^1_\M(L^I_1(M_i))_{-1} \cong H^1_\M(L_0^I(M_{i+1}))_{-1}$ and 
$H^1_\M(L^I_1(M_i))_{j} = 0$ for $j \geq 0$. 

The exact sequence $0 \rt M_{i+1} \rt F_i \rt M_i \rt 0$ induces
 an exact sequence $0 \rt L^I_1(M_i) \rt L^I_0(M_{i+1}) \rt L^I_0(F_i) \xrightarrow{\pi_i}  L^I_0(M_i) \rt 0$. Let $C = \ker \pi_i$.

As $G_I(A)$ is \CM \ we have $H^i_\M(L^I_0(A)) = 0$ for $i = 0, 1$, see \cite[5.2]{Pu2}. Thus 
$H^0_\M(C) = 0$ and  $H^1_\M(C) \cong H^0_\M(L^I_0(M_i))$. In particular we have $H^1_\M(C)_{-1} = 0$.

 The exact sequence $0 \rt L^I_1(M_i) \rt L^I_0(M_{i+1}) \rt C \rt 0$ induces an exact sequence
 \[
 0 \rt H^1_\M(L^I_1(M_i)) \rt H^1_\M(L^I_0(M_{i+1})) \rt H^1_\M(C)
 \]
 As $H^1_\M(C)_{-1} = 0$ we get $H^1_\M(L^I_1(M_i))_{-1} \cong H^1_\M(L_0^I(M_{i+1}))_{-1}$. As \\  $H^1_\M(L^I_0(M_{i+1}))_j = 0$ for $j \geq 0$ we also get 
 $H^1_\M(L^I_1(M_i))_{j} = 0$ for $j \geq 0$. Thus the Claim is proved.
 
 Let $x$ be $M_i \oplus A$-superficial for each $i \geq 0$. Set $X = xu$,  $\ov{A} = A/(x)$ and $\ov{M_i} = M_i/xM_i$.  As $x^*$ is $G_I(A)$-regular, we have an exact sequence
 \[
 0 \rt L^I_0(A)(-1) \xrightarrow{X}  L^I_0(A) \rt  L^I_0(\ov{A}) \rt 0.
 \]
 After tensoring with $M$ this induces for $i \geq 0$ an exact sequence
 \[
 L^I_{i+1}(M)(-1)  \xrightarrow{X} L^I_{i+1}(M) \rt L^I_{i+1}(\ov{M}) \rt 
 L^I_{i}(M)(-1)  \xrightarrow{X} L^I_{i}(M).
 \]
 We note that $L^I_{i+1}(M) \cong L^I_1(M_i)$. Also by \ref{H1} we get that
 $H_1(X,  L^I_1(M_i)) \cong H_1(X, L^I_0(M_{i+1}))$.  Thus we have an exact sequence
 \begin{align*}
 0 &\rt H_1(X, L^I_0(M_{i+1})) \rt  L^I_1(M_i)(-1)  \xrightarrow{X}  L^I_1(M_i) \rt \\
 &\xrightarrow{\alpha_i} L^I_1(\ov{M_i}) \rt  H_1(X, L^I_0(M_{i})) \rt 0.
 \end{align*}

 Let $W_i = \ker \alpha_i$ and $D_i = \image \alpha_i$. As $H_1(X, L^I_0(M_{i+1}))$ 
 has finite length we get a surjection $H^0_\M(L^I_1(M_i))(-1) \rt H^0_\M(W_i)$ and a
 isomorphism $H^1_\M(W_i) \cong   H^1_\M(L^I_1(M_i))(-1)$. Thus 
\begin{equation}\label{5.5}
H^0_\M(W_i)_0 = 0 \quad \text{and} \quad H^1_\M(W_i)_{0} \cong  H^1_\M(L^I_1(M_i))_{-1}.
\end{equation}

 The exact sequence $ 0 \rt W_i \rt L^I_1(M_i) \rt D_i \rt 0$ induces an exact sequence
 \[
 0 \rt H^0_\M(W_i) \rt H^0_\M(L^I_1(M_i)) \rt H^0_\M(D_i) \rt H^1_\M(W_i) \rt
 H^1_\M(L^I_1(M_i))
 \]
 Evaluating at $n = 0$ we get an exact sequence
 \begin{equation}\label{six}
  0 \rt H^0_\M(L^I_1(M_i))_0 \rt H^0_\M(D_i)_0 \rt H^1_\M(W_i)_0 \rt 0.
 \end{equation}

 By the exact sequence $0 \rt D_i \rt L^I_1(\ov{M_i}) \rt  H_1(X, L^I_0(M_{i})) \rt 0$ we obtain an exact sequence
 \[
 0 \rt H^0_\M(D_i) \rt H^0_\M(L^I_1(\ov{M_i})) \rt H_1(X, L^I_0(M_{i}))
 \]
 As $H_1(X, L^I_0(M_{i}))_0 = 0$ we get $H^0_\M(D_i)_0  \cong H^0_\M(L^I_1(\ov{M_i}))_0$.
 By (\ref{six}) we get that
 \[
 \ell(H^0_\M(L^I_1(M_i))_0) \leq \ell(H^0_\M(L^I_1(\ov{M_i}))_0) \quad \text{with equality iff} \  H^1_\M(W_i)_0  = 0.
 \]
 The latter condition holds by (\ref{5.5}) if and only if $H^1_\M(L^I_1(M_i))_{-1} = 0$ and this by our claim holds if and only if $H^1_\M(L^I_0(M_{i+1}))_{-1} = 0$. This holds if and only if $\xi_I(M_{i+1}) \geq 2$.
 
 We now claim that the functions $i \mapsto \ell(H^0_\M(L^I_1(M_i))_0)$ and 
 $i \mapsto \ell(H^0_\M(L^I_1(\ov{M_i}))_0)$ are of quasi-polynomial type with period two. This will prove our assertion. To see that these functions are of quasi-polynomial type, let $t_1,\ldots, t_c$ be the Eisenbud operators over $\mathbb{F}$. Then $L(M) = \bigoplus_{i \geq 0}L^I_i(M)$ is a bigraded module over $S = \R[t_1,\ldots,t_c]$, see \ref{bi}. Thus $H^0_\M(L(M)) = \bigoplus_{i \geq 0}H^0_\M(L^I_i(M))$ is a bigraded module over $S$. By \ref{RR} we get that $H^0_\M(L^I_i(M))_0 \cong \wt{IM_i}/IM_i$. As $\bigoplus_{i \geq 0}L^I_i(M)_0$ is an Artin $A[t_1,\ldots, t_c]$-module we get that $\bigoplus_{n\geq 0}H^0_\M(L^I_i(M))_0  $ is an Artin $A[t_1,\ldots, t_c]$-module. It follows that the function  $i \mapsto \ell(H^0_\M(L^I_1(M_i))_0)$ is of quasi-polynomial type with period two. A similar argument yields that the function \\ $i \mapsto \ell(H^0_\M(L^I_1(\ov{M_i}))_0)$ is of quasi-polynomial type with period two.
\end{proof}

\section{Dual Hilbert Coefficients}
In this section we assume that $(A,\m)$ is a \CM \ local ring with a canonical module $\omega$. Let $I$ be an $\m$-primary ideal and let $M$ be a maximal \CM \ $A$-module. The function $D^I(M,n) = \ell( \Hom_A(M, \omega/I^{n+1}\omega)$ is called the \emph{dual Hilbert-Samuel} function of $M$ \wrt \  $I$. In \cite{PuZ} it is shown that there exist a polynomial $t^I(M,z) \in \mathbb{Q}[z]$ of degree $d$ such that $t^I(M,n) = D^I(M,n)$ for all $n \gg 0$.
We write
$$ t^I(M,X) = \sum_{i = 0}^{d}(-1)^ic_i^I(M)\binom{X+ r -i}{r-i}.$$
The integers $c_i^I(M)$ are called the $i^{th}$- dual Hilbert coefficient of $M$ with respect to $I$. The zeroth dual Hilbert coefficient $c_0^I(M)$ is equal to $e_0^I(M)$.
We prove:
\begin{theorem}\label{main-dual-2}
Let $(A,\m)$ be a \CM  \ local ring  of dimension $d$, with a canonical module $\omega$ and let $I$ be an $\m$-primary ideal. Let $M$ be a maximal \CM \ $A$-module. Assume $M$ has finite GCI dimension.  Then
for $i = 0,1, \cdots, d$  the function  $j \mapsto c_i^I(\Syz^A_j(M))$ 
is of quasi-polynomial type with period two. If $A$ is a complete intersection then the degree of each of the above functions $\leq \cx(M)-1$.
\end{theorem}
\begin{proof}
By Remark \ref{base-change} we may
 assume $A = Q/(f_1,\ldots, f_c)$ where $Q$ is complete with uncountable residue field, $\mathbf{f} = f_1,\cdots, f_c$ is a $Q$-regular sequence and 
$\projdim_Q M$ is finite.  
Let $\mathbb{F}$ be a minimal resolution of $M$. Let $t_1,\ldots,t_c$ be the Eisenbud operators over $\mathbb{F}$.
Let $\R = A[ut]$ be the Rees algebra of $A$ \wrt \ $I$ and let $N = \bigoplus_{n \geq 0} N_n$ is a finitely generated $\R$-module. Give
$E(N) = \bigoplus_{i,n \geq 0} \Ext^i_A(M,N_n)$ a bi-graded $S = \R[t_1,\ldots,t_c]$
structure as described in \ref{bigraded}. Then by \cite[1.1]{Pu3}, $E(N)$ is a finitely generated $S$-module. In particular $E_I(\omega, M) = \bigoplus_{i,n \geq 0} \Ext^i_A(M,I^n\omega)$ is a finitely generated $S$-module. We now note that \\
$\Ext^{i}_A(M, \omega/I^{n+1}\omega) \cong \Ext^{i+1}_A(M,I^{n+1}\omega)$ for all $i \geq 1$ and for all $n \geq 0$. Thus 
$$D_I(\omega, M) = \bigoplus_{i \geq 1,n \geq 0}\Ext^{i}_A(M, \omega/I^{n+1}\omega)$$
is a submodule of $E_I(\omega, M)(0,1)$ and so is finitely generated. Set \\ $K = \ann_S D_I(\omega, M)$. Then $\q = K_{0,0}$ is $\m$-primary and so in particular $D_I(\omega, M)$ is a finitely generated $T = \R/\q \R[t_1,\ldots,t_c]$-module. Therefore the Hilbert series of  $D_I(\omega, M)$ is of the form
\[
\frac{h(z,w)}{(1-z)^d(1-w^2)^c}
\]
By \ref{g2-growth} we get that for $i, n \gg 0$
\[
\ell(\Ext^{i}_A(M, \omega/I^{n+1}\omega)) = \sum_{l = 0}^{d-1}(-1)^lv_l^I(i)\binom{n+d-1-l}{d-1-l}
\]
where the functions $i \mapsto v_l^I(i)$ for $i = 0,\ldots,d-1$ are of quasi-polynomial type with period two and degree $\leq c -1$.  We note that it is possible that some of the $v_l^I(i)$ is identically zero.\\
Now set $M_i = \Syz^A_i(M)$. The exact sequence $0 \rt M_{i+1} \rt F_i \rt M_i \rt 0$ induces an exact sequence
\begin{align*}
0 \rt &\Hom_A(M_i, \omega/I^{n+1}\omega) \rt \Hom_A(F_i, \omega/I^{n+1}\omega) \rt \Hom_A(M_{i+1}, \omega/I^{n+1}\omega)\\
\rt &\Ext^{i+1}_A(M, \omega/I^{n+1}\omega) \rt 0.
\end{align*}
Thus for $1 \leq l \leq  d$ we have that
\[
\rank(F_i)e_l^I(\omega) - c_l^I(M_i) - c_l^I(M_{i+1}) = v_{l-1}^I(i+1).
\]
Using \ref{g2-c} we get that for $1 \leq l \leq d$ the function $i \mapsto c_l^I(M_i)$ is of quasi-polynomial type with period two and degree $\leq c-1$. Also note that $c_0^I(M_i) = e_0^I(M_i)$. So for $l = 0$ the result follows from \ref{mult}.

If $A$ is a complete intersection then we may take $Q$ to be a complete regular local ring. Then by \cite[3.9]{LLAV} there exists a complete local ring $R$ with $A = R/(g_1,\ldots,g_r)$ with $r = \cx(M)$,  $g_1,\ldots,g_r$ a $R$-regular sequence and $\projdim_R M$ finite. It now follows that the degree of the functions  $i \mapsto c_l^I(M_i)$ is $\leq \cx(M) -1$.
\end{proof}

\section{regularity}
Let $H^i(-)$ denote the $i^{th}$ local cohomology functor of $G_I(A)$ \wrt \ $G_I(A)_+ = \bigoplus_{n > 0}I^n/I^{n+1}$. 
Set
\[
a_i(G_I(M)) = \max \{ j \mid H^i(G_I(M))_j \neq 0 \}.
\]
Assume that the residue field of $A$ is infinite. Let $J$ be a minimal reduction of $I$. Say $I^{r+1} = JI^r$. Let $M$ be a maximal \CM \ $A$-module. Then it is well-known that $a_d(G_I(M))  \leq r - d $. We prove
\begin{theorem} \label{main-reg-2}
Let $(A,\m)$ be a \CM \ local ring of dimension $d \geq 2 $ and let $M$ be a maximal \CM \ $A$-module.  Let $I$ be an $\m$-primary  ideal. Assume $M$ has finite GCI dimension. Then the set
\[
\left\{ \frac{a_{d-1}(G_I(\Syz^A_i(M))}{i^{\cx(M)-1}}  \right\}_{i \geq 1 }
\]
is bounded.
\end{theorem}
\begin{proof}
By Remark \ref{base-change} we may
 assume $A = Q/(f_1,\ldots, f_c)$ where $Q$ is complete with uncountable residue field, $\mathbf{f} = f_1,\cdots, f_c$ is a $Q$-regular sequence and 
$\projdim_Q M$ is finite.  We prove the result by induction on $d = \dim A$. 

We first consider the case when $d = 2$. Set $M_i = \Syz^A_i(M)$ for $i \geq 0$. As $k$ is uncountable we can choose
$x,y$  an $M_i \oplus A$-superficial sequence for all $i \geq 0$. Set $J = (x,y)$. Then $I^{r+1} = JI^r$ for some $r \geq 0$. Let $\R = A[ut]$ be the Rees algebra of $A$ \wrt \ $I$ and let $\M$ be its unique maximal homogeneous ideal.  We note that for all $i \geq 0$ we have an isomorphism of $\R$-modules
$H^i_\M(G_I(E)) \cong H^i_{G_I(A)_+}(G_I(E))$ for any finitely generated $A$-module $E$. The exact sequence of $\R$-modules $0 \rt G_I(M_i) \rt L^I_0(M_i) \rt L^I_0(M_i)(-1) \rt 0$ induces an exact sequence for all $n \in \Z$
\begin{align*}
H^0_\M(L^I_0(M_i))_{n-1} &\rt H^1(G_I(M_i))_n \rt H^1_\M(L^I_0(M_i))_n \rt H^1_\M(L^I_0(M_i))_{n-1} \rt
\\
&\rt  H^2(G_I(M_i))_n.
\end{align*}
As $H^2(G_I(M_i))_n = 0$ for $n \geq r-1$ we get that $H^I_\M(L^I_0(M_i))_n = 0$ for 
$n \geq r - 2$. 
Set $\rho^I(M_i) = \min\{ j \mid \wt{I^nM_i} = I^nM_i \ \text{for all} \ n \geq j \}$.
As $H^0_\M(L^I_0(M_i)) = \bigoplus_{n\geq 0}\wt{I^{n+1}M}/I^{n+1}M$, to prove our result it suffices to show that
\[
\left\{\frac{\rho^I(M_i)}{i^{\cx(M)-1}}  \right\}_{i \geq 1 }  \quad \text{is bounded.}
\]
We have nothing to show if $\cx(M) \leq 1$ for then  either $M$ is free (if $\cx(M) = 0$) or$M$ has a periodic resolution with period two (if $\cx(M) = 1$). So assume $\cx(M) \geq 2$.

 Set 
\[
B^I(M_i) = \bigoplus_{n \geq 0} \frac{I^{n+1}M_i \colon x}{I^nM_i} = H_1(xu, L^I_0(M_i)).
\]
Set $\ov{M_i} = M_i/xM_i$.
By  \cite[2.9]{Pu2} we have an exact sequence
\[
0 \rt \frac{I^{n+1}M_i \colon x}{I^nM_i} \rt \frac{\wt{I^nM_i}}{I^nM_i} \rt
\frac{\wt{I^{n+1}M_i}}{I^{n+1}M_i} \rt \frac{\wt{I^{n+1}\ov{M_i}}}{I^{n+1}\ov{M_i}}. 
\]
By Lemma \ref{crucial} we get that $\wt{I^n\ov{M_i}} = I^n\ov{M_i}$ for all $n \geq r$. By the above exact sequence we get that 
\[
\ell\left( \frac{\wt{I^rM_i}}{I^rM_i} \right) = \sum_{n \geq r} \ell \left(  \frac{I^{n+1}M_i \colon x}{I^nM_i} \right) \leq \ell(B^I(M_i)),  \text{and} 
\]
\[
\ell\left(\frac{\wt{I^nM_i}}{I^nM_i} \right) \geq  \ell\left( \frac{\wt{I^{n+1}M_i}}{I^{n+1}M_i} \right)  \quad \text{for} \  n \geq r.
\]
Claim: If $n \geq r$ and $\wt{I^nM_i} \neq I^nM_i$ then we have a strict inequality
\[
\ell\left(\frac{\wt{I^nM_i}}{I^nM_i} \right) >  \ell\left( \frac{\wt{I^{n+1}M_i}}{I^{n+1}M_i} \right)  
\]
Proof of Claim: If the result does not hold then we have $(I^{n+1}M_i \colon x) = I^nM_i$. For all $m \geq 1$ we have an exact sequence
\[
  \frac{(I^{m}M_i \colon x)}{I^{m-1}M_i}      \xrightarrow{\beta_m}         \frac{(I^{m+1}M_i \colon x)}{I^mM_i}      \xrightarrow{\alpha_m}    \frac{I^{m+1}M_i}{JI^mM_i}     \xrightarrow{\rho_m}  \frac{I^{m+1}\ov{M_i}}{yI^m\ov{M_i}}\rt 0,
\]
(see proof of Lemma \ref{depth1}).
So if $n \geq r$ and $(I^{n+1}M_i \colon x) = I^nM_i$ then $(I^{m+1}M_i \colon x) = I^mM_i$ for all $m \geq n$. So we get
\[
\ell\left( \frac{\wt{I^nM_i}}{I^nM_i} \right) = \sum_{m \geq n} \ell \left(  \frac{I^{m+1}M_i \colon x}{I^mM_i} \right) = 0, \quad \text{a contradiction}.
\]
It follows that
\[
\rho^I(M_i) \leq r + 1 + \ell(B^I(M_i)).
\]
We consider the following two cases.

Case I: $\depth G_I(A) > 0$. So $x^*$ is $G_I(A)$-regular. By proof of Theorem \ref{e2} the function $i \mapsto \ell(B^I(M_i))$ is of quasi-polynomial type with period two and degree $\leq \cx(M)-1$.  

It follows that 
\[
\left\{\frac{\rho^I(M_i)}{i^{\cx(M)-1}}  \right\}_{i \geq 1 }  \quad \text{is bounded.}
\]

Case II. $\depth G_I(A) = 0$.
We note that $\depth G_{I^m}(A) \geq 1$ for all $m \gg 0$. Choose $s$ such that
$\depth G_{I^s}(A) \geq 1$. For $n \geq r$ we have
\[
\ell\left(\frac{\wt{I^nM_i}}{I^nM_i} \right) \geq  \ell\left( \frac{\wt{I^{n+1}M_i}}{I^{n+1}M_i} \right).
\] 
It follows that 
\[
\rho^I(M_i) \leq \max \{ r, s\rho^{I^s}(M_i) \}.
\]
By our Case I the result follows.

Now assume $d \geq 3$ and that the result holds when dimension of the ring is $= d -1$. 
Let $x$ be $M_i \oplus A$-superficial for all $i \geq 0$. Set $\ov{M_i} = M_i/xM_i$.
Then notice $\ov{M}$ has finite GCI-dimension over $\ov{A}$ and $\ov{M_i} \cong \Syz^{\ov{A}}_i(\ov{M})$ for $i \geq 0$. We have exact sequences
\begin{align*}
0 \rt U_i &\rt G_I(M_i)(-1)\xrightarrow{x^*} G_I(M_i) \rt G_I(M_i)/x^*G_I(M_i) \rt 0   \ \ \text{and} \\
0\rt V_i &\rt G_I(M_i)/x^*G_I(M_i) \rt G_I(\ov{M_i}) \rt 0;
\end{align*}
where $U_i, V_i$ are $G_I(A)$-modules of finite length. As $d \geq 3$ we have an exact sequence for all $n \in \Z$
\[
H^{d-2}(G_I(\ov{M_i}))_n \rt H^{d-1}(G_I(M_i))_{n-1} \rt H^{d-1}(G_I(M_i))_{n}
\]
It follows that
\[
a_{d-1}(G_I(M_i)) \leq a_{d-2}(G_I(\ov{M_i})) - 1.
\]
The result now holds by induction hypothesis.
\end{proof}


\begin{thebibliography}{10}

\bibitem{LLAV}
L.~L. Avramov, \emph{Modules of finite virtual projective dimension},
  Invent.~math \textbf{96} (1989), 71--101.
  
  \bibitem{AGP}
 L.~L.~Avramov, V.~N.~Gasharov and I.~.V. Peeva, \emph{Complete intersection
  dimension}, Inst. Hautes \'Etudes Sci. Publ. Math. (1997), no.~86, 67--114
  (1998).

\bibitem{BS}  
M.~P.~Brodmann and R.~Y.~Sharp, 
\emph{Local cohomology: an algebraic introduction with geometric applications}, Cambridge Studies in Advanced Mathematics 60, Cambridge University Press,~Cambridge, 1998.

\bibitem{BH}
W.~Bruns and J.~Herzog, \emph{{Cohen-Macaulay rings}}, 
 Edition, vol.~39, Cambridge
  studies in advanced mathematics, Cambridge University Press,~Cambridge, 1997.
  
\bibitem{Eisenbud-80}
David Eisenbud, \emph{Homological algebra on a complete intersection, with an
  application to group representations}, Trans. Amer. Math. Soc. \textbf{260}
  (1980), no.~1, 35--64. 
  
\bibitem{Elias}  
  J.~Elias, 
\emph{Depth of higher associated graded rings},
J. London Math. Soc. (2) 70 (2004), no. 1, 41–-58. 

\bibitem{Gulliksen}
Tor~H. Gulliksen, \emph{A change of ring theorem with applications to
  {P}oincar\'e series and intersection multiplicity}, Math. Scand. \textbf{34}
  (1974), 167--183. 
  
  \bibitem{Pu1}
 T.~J.~Puthenpurakal, \emph{Hilbert-coefficients of a Cohen-Macaulay module}, J. Algebra 264 (2003), no. 1, 82–-97.
 
\bibitem{PuMCM}
\bysame,  
 \emph{The Hilbert function of a maximal Cohen-Macaulay module}, 
Math. Z. 251 (2005), no. 3, 551–-573. 
  
  \bibitem{Pu2}
   \bysame, \emph{Ratliff-Rush filtration, regularity and depth of higher associated graded modules. I}, J. Pure Appl. Algebra 208 (2007), no. 1, 159–-176.

\bibitem{Pu3}
   \bysame, \emph{On the finite generation of a family of Ext modules},
    Pacific J. Math. 266 (2013), no. 2, 367–-389. 

\bibitem{PuZ}
   \bysame and F.~Zulfeqarr, \emph{The dual Hilbert-Samuel function of a Maximal Cohen-Macaulay module}, to appear in Communications in Algebra, eprint: arxiv: 0809.3353
 
\bibitem{T}
 N. V. Trung, 
 \emph{Reduction exponent and degree bound for the defining equations of graded rings},
  Proc. Amer. Math. Soc. 101(1987), no. 2, 229–-236.
 
  
  
  
\end{thebibliography}
\end{document}